\documentclass[12pt]{article}
\usepackage{latexsym, amssymb, amsmath, amscd, amsfonts, epsfig, graphicx, colordvi,verbatim,ifpdf}
\usepackage{amsfonts, amsmath, amssymb,extarrows}
\usepackage{amssymb,amsfonts,amsmath,latexsym,epsfig,cite, psfrag,eepic,color}
\usepackage{amscd,graphics}
\usepackage{latexsym, amssymb,  amsmath,amscd, amsfonts, epsfig, graphicx, colordvi,amsthm}
\usepackage[colorlinks,urlcolor=purple,linkcolor=red,anchorcolor=green,citecolor=blue]{hyperref}

\usepackage{graphicx}
\usepackage{color}
\usepackage{ifpdf}
\usepackage{fancybox}

\usepackage{float}

\newtheorem{thm}{Theorem}[section]

\newtheorem{rem}[thm]{\it Remark}

\newtheorem{lem}[thm]{Lemma}

\def\pf{\noindent{\it Proof.} }
\def\qed{\nopagebreak\hfill{\rule{4pt}{7pt}}
\medbreak}

\numberwithin{equation}{section}

\def\qed{\nopagebreak\hfill{\rule{4pt}{7pt}}
\medbreak}

\newlength{\boxedparwidth}
  {\begin{center} \begin{tabular}{|@{\hspace{.315in}}c@{\hspace{.15in}}|}
                  \hline \\ \begin{minipage}[t]{\boxedparwidth}
                  \setlength{\parindent}{.25in}}%
  {\end{minipage} \\ \\ \hline \end{tabular} \end{center}}

\allowdisplaybreaks
\begin{document}

\begin{center}

 {\large \bf Inequalities for the Overpartition Function}
\end{center}
\vskip 0.2cm

\begin{center}
{Edward Y.S. Liu}$^{1}$ and {Helen W.J. Zhang}$^{2}$\vskip 3mm

  $^{1}$Center for Combinatorics, LPMC\\[2pt]
Nankai University, Tianjin 300071, P.R. China\\[5pt]

   \vskip 2mm

     $^{2}$Center for Applied Mathematics\\[2pt]
   Tianjin University, Tianjin 300072, P.R. China\\[8pt]
  $^1$liu@mail.nankai.edu.cn, \  $^2$wenjingzhang@tju.edu.cn

\end{center}

\vskip 6mm \noindent {\bf Abstract.}
Let $\overline{p}(n)$ denote the overpartition funtion. Engel showed that for $n\geq2$, $\overline{p}(n)$ satisfied the Tur\'{a}n inequalities, that is, $\overline{p}(n)^2-\overline{p}(n-1)\overline{p}(n+1)>0$ for $n\geq2$. In this paper, we prove several inequalities for $\overline{p}(n)$.
Moreover, motivated by the work of Chen, Jia and Wang, we find that the higher order Tur\'{a}n inequalities of $\overline{p}(n)$ can also be determined.

\vskip 0.3cm

\noindent {\bf Keywords}: overpartition function, Rademacher-type series, log-concavity,  higher order Tur\'{a}n inequalities
\vskip 0.3cm

\noindent {\bf AMS Classifications}: 05A20, 11P82, 11P99 \vskip 0.3cm

\section{Introduction}

A partition of a positive integer $n$ is a non-increasing sequence of positive integers whose sum is $n$. Let $p(n)$ denote the number of partitions of $n$.
Recall that a sequence $\{a_i\}_{0\leq i\leq n}$ is said to satisfy the Tur\'{a}n inequalities if
\[a_i^2-a_{i+1}a_{i-1}\geq0,~~1\leq i\leq n.\]
In particular, a sequence satisfying the Tur\'{a}n inequalities can also be called log-concave.
DeSalvo and Pak \cite{DeSalvo-Pak-2015} showed that $p(n)$ is log-concave for all $n\geq25$. They also proved two conjectures given by Chen \cite{Chen-2010},
\[\frac{p(n-1)}{p(n)}\left(1+\frac{1}{n}\right)>\frac{p(n)}{p(n+1)},~~~\text{for}~n>1,\]
and
\[p(n)^2-p(n-m)p(n+m)\geq0,~~~\text{for}~n>m>1.\]
Since then, the inequalities between the partition functions have been extensively studied. For example, Chen, Wang and Xie \cite{Chen-Wang-Xie-2016} proved a sharper inequality
\[\frac{p(n-1)}{p(n)}\left(1+\frac{\pi}{\sqrt{24}n^{3/2}}\right)
>\frac{p(n)}{p(n+1)}\]
holds for $n\geq45$, which was conjectured by Desalvo and Pak \cite{DeSalvo-Pak-2015}.
Bessenrodt and Ono \cite{Bessenrodt-Ono-2016} obtained that
\[p(a)p(b)\geq p(a+b)\]
holds for $a,b>1$ and $a+b>8$. Based on this inequality, they extended the partition function multiplicatively to a functions on partitions and showed that it has a unique maximum at an explicit partition for any $n\neq 7$. Recently, Dawsey and Masri \cite{Dawsey-Masri-2017} gave an effective asymptotic formula of the Andrews spt-function due to the algebraic formula \cite{Ahlgren-Andersen-2016} for the spt-function. According to this asymptotic formula, they proved some inequalities on the spt-function conjectured by Chen \cite{Chen-2018}.

The similar inequalities can also be satisfied by the overpartition function.
Recall an overpartition \cite{Corteel-Lovejoy-2004} of a nonnegative integer $n$ is a partition of $n$ where the first occurrence of each distinct part may be overlined.
Let $\overline{p}(n)$ denote the number of overpartitions of $n$.
Zukermann  \cite{Zuckerman-1939} gave a formula for the overpartition function, which is indeed a Rademacher-type convergent series,
\begin{align}\label{overlinep-asym}
\overline{p}(n)=\frac{1}{2\pi}\sum_{k=1\atop 2\nmid k}^\infty\sqrt{k}\sum_{h=0\atop (h,k)=1}^k
\frac{\omega(h,k)^2}{\omega(2h,k)}e^{-\frac{2\pi inh}{k}}\frac{d}{dn}
\left(\frac{\sinh\frac{\pi\sqrt{n}}{k}}{\sqrt{n}}\right),
\end{align}
where
\[\omega(h,k)=\exp\left(\pi i\sum_{r=1}^{k-1}\frac{r}{k}\left(\frac{hr}{k}
-\left\lfloor\frac{hr}{k}\right\rfloor-\frac{1}{2}\right)\right)\]
for positive integers $h$ and $k$.
Let $\mu=\mu(n)=\pi\sqrt{n}$. From this Rademacher-type series \eqref{overlinep-asym}, Engel \cite{Engel-2017} provided an error term for the overpartition function
\begin{align*}
\overline{p}(n)=\frac{1}{2\pi}\sum_{k=1\atop 2\nmid k}^N\sqrt{k}\sum_{h=0\atop (h,k)=1}^k
\frac{\omega(h,k)^2}{\omega(2h,k)}e^{-\frac{2\pi inh}{k}}\frac{d}{dn}
\left(\frac{\sinh\frac{\mu}{k}}{\sqrt{n}}\right)+R_2(n,N),
\end{align*}
where
\begin{align}\label{R_2(n,N)}
|R_2(n,N)|\leq \frac{N^{\frac{5}{2}}}{n\mu}
\sinh\left(\frac{\mu}{N}\right).
\end{align}
In particular, when $N=3$, we have
\begin{align}\label{overlinep-asym-1}
\overline{p}(n)=\frac{1}{8n}\left[\left(1+\frac{1}{\mu}\right)e^{-\mu}+\left(1-\frac{1}{\mu}\right)e^{\mu}\right]
+R_2(n,3),
\end{align}
where
\begin{align}\label{R_2(n,3)}
|R_2(n,3)|\leq \frac{3^{\frac{5}{2}}}{n\mu}
\sinh\left(\frac{\mu}{3}\right).
\end{align}
Moreover, using this asymptotic formula \eqref{overlinep-asym-1}, Engel \cite{Engel-2017} proved that $\overline{p}(n)$ is log-concave for $n\geq2$, that is,
\begin{align}\label{overlinep-log-con}
\overline{p}(n)^2>\overline{p}(n-1)\overline{p}(n+1).
\end{align}

Let $\Delta$ be the difference operator as given by $\Delta f(n)=f(n+1)-f(n)$. Recently, Wang, Xie and Zhang \cite{Wang-Xie-Zhang-2018} showed that for any given $r\geq1$, there exists a positive number $n(r)$ such that $(-1)^{r-1}\Delta^r\log\overline{p}(n)>0$ for $n>n(r)$. Moreover, they gave an upper bound for $(-1)^{r-1}\Delta^r\log\overline{p}(n)$. More precisely, for all $r\geq 1$, there exists a positive integer $n(r)$ such that for $n>n(r)$,
\[(-1)^{r-1}\Delta^r\log\overline{p}(n)<\frac{\pi}{2}\left(\frac{1}{2}\right)_{r-1}\frac{1}{n^{r-1/2}}.\]
where $\left(x\right)_{n}:=x\cdot(x+1)\cdots(x+n-1)$.
From the proof of \cite[Theorem 4.1]{Wang-Xie-Zhang-2018}, we can obtain a slight modification of this result as follows
\[(-1)^{r-1}\Delta^r\log\overline{p}(n-1)<\frac{\pi}{2}\left(\frac{1}{2}\right)_{r-1}\frac{1}{n^{r-1/2}}.\]
In particular, when $r=2$, we have
\begin{align}\label{exact-logconcave}
\frac{\overline{p}(n-1)}{\overline{p}(n)}\left(1+\frac{\pi}{4n^{3/2}}\right)
\geq\frac{\overline{p}(n)}{\overline{p}(n+1)},\quad\text{for}\, n\geq2.
\end{align}

In this paper, we prove some inequalities for the overpartition function.
One of main results of this paper is the following theorem analogues to these equalities for the partition function obtained by DeSalvo and Pak \cite{DeSalvo-Pak-2015}, Bessenrodt and Ono \cite{Bessenrodt-Ono-2016}.

\begin{thm}\label{overlinep-ineq-thm}

\begin{description}
\item[{\rm(1)}]
For all $n>m>1$, we have
\begin{align}\label{overlinep(n-m)}
\overline{p}(n)^2-\overline{p}(n-m)\overline{p}(n+m)\geq0,
\end{align}
with equality holding only for $(n,m)=(2,1)$.

\item[{\rm(2)}]
If $a$, $b$ are integers with $a, b>1$, then
\begin{align}\label{p(a)p(b)}
\overline{p}(a)\overline{p}(b)>\overline{p}(a+b).
\end{align}

\end{description}

\end{thm}
To state the second result, we first introduce some definitions.
Given a function $\gamma:\mathbb{N}\mapsto\mathbb{R}$ and positive integers $d$ and $n$, the associated
Jensen polynomial of degree $d$ and shift $n$ is defined by
\[J^{d,n}_\gamma(n):=\sum_{j=0}^d\binom{d}{j}\gamma_{n+j}X^j.\]
If all of zeros of a polynomial are real, then this polynomial is said to be hyperbolic.
A real entire function
\[\psi(x)=\sum_{k=0}^\infty \gamma_k\frac{x^k}{k!}\]
is said to be in the Laguerre-P\'{o}lya class if it can be
represented in the form
\[\psi(x)=cx^ne^{-\alpha x^2+\beta x}\prod_{k=1}^\infty(1+x/x_k)e^{-x/x_k},\]
where $c$, $\beta$, $x_k$ are real numbers, $\alpha\geq 0$, $m$ is a nonnegative integer and
$\sum x_k^{-2}<\infty$.
Jensen \cite{Jensen-1913} showed that $\psi(x)$ belongs to the Laguerre-P\'{o}lya class if and only if all of  the associated Jensen polynomials $J^{d,0}_{\gamma }$ are hyperbolic. P\'{o}lya \cite{Polya-1927}  proved that the Riemann Hypothesis is equivalent to the hyperbolicity of all Jensen polynomials associated to Riemann's $\xi$-function.

 The Tur\'{a}n inequalities and the higher order Tur\'{a}n inequalities are related to the Lagurre-P\'{o}lya class of real entire functions. From the work of P\'{o}lya and Schur \cite{Polya-Schur-1914} we see that the Maclaurin coefficients of $\psi(x)$ in the Lagurre-P\'{o}lya class satisfy the Tur\'{a}n inequalities
 \[\gamma_k^2-\gamma_{k-1}\gamma_{k+1}\geq0\]
 for $k\geq 1$.
 Due to the result of Dimitrov \cite{Dimitrov}, we know that the Macalurin coefficients of $\psi(x)$ in the Lagurre-P\'{o}lya class satisfy the higher order Tur\'{a}n inequalities
 \[4(\gamma_k^2-\gamma_{k-1}\gamma_{k+1})(\gamma^2_{k+1}-\gamma_k\gamma_{k+2})
 -(\gamma_k\gamma_{k+1}-\gamma_{k-1}\gamma_{k+2})^2\geq0\]
 for $k\geq1$.

Clearly, from the results of Desalvo and Pak \cite{DeSalvo-Pak-2015},
Engel \cite{Engel-2017}
and Dawsey and Masri \cite{Dawsey-Masri-2017}, we see that
the partition function, the overpartition function and the spt-function all satisfied the Tur\'{a}n inequalities.
Moreover,
Chen, Jia and Wang \cite{Chen-Jia-Wang-2017} showed that the partition function satisfied the higher order Tur\'{a}n inequalities. In this paper, we confirm the overpartition function also satisfied the higher order Tur\'{a}n inequalities.
\begin{thm}\label{turan-ineq}
Let
\begin{align*}
u_n=\frac{\overline{p}(n-1)\overline{p}(n+1)}{\overline{p}(n)^2}.
\end{align*}
For $n\geq16$,
\begin{align*}
4(1-u_n)(1-u_{n+1})-(1-u_nu_{n+1})^2>0.
\end{align*}
\end{thm}

\begin{rem}
Recently, Griffin, Ono, Rolen and Zagier \cite{G-O-R-Z-2018} proved that Jensen polynomials for weakly holomorphic modular forms on $SL_2(\mathbb{Z})$ with real coefficients and a pole at $i\infty$ are eventually hyperbolic.
This work proved Chen, Jia and Wang's conjecture \cite{Chen-Jia-Wang-2017} that the Jensen polynomials associated to the partition function $p(n)$ are eventually hyperbolic as a special case. In other words, for each $d\geq 1$ there exists some $N(d)$ such that for all $n\geq N(d)$, the polynomial $J_p^{d,n}(x)$ is hyperbolic.
 Larson and Wagner \cite{Larson-Wagner-2018}  computed the values of the minimal $N(d)$ for $d=3,4,5$ and gave an upper bound of the minimal $N(d)$ for each $d\geq 1$. Moreover, the work of Griffin, Ono, Rolen and Zagier \cite{G-O-R-Z-2018}  can also be used to prove that the Jensen polynomials associated to the overpartition function $\overline{p}(n)$  are eventually hyperbolic. In this paper, we give an explicit bound  for the Jensen polynomial $J_{\overline{p}}^{3,n}(x)$, that is, for all $n\geq 16$,  $J_{\overline{p}}^{3,n}(x)$ is hyperbolic.
  \end{rem}

\section{Proof of Theorem \ref{overlinep-ineq-thm}}

In this section we give a proof of Theorem \ref{overlinep-ineq-thm}.
To prove the part (1) of Theorem \ref{overlinep-ineq-thm}, we need the following lemma, which is looser but more simple than \eqref{overlinep-asym-1} and \eqref{R_2(n,3)}.

\begin{lem}\label{overlinep-asym-lem}
For all $n\geq1$, we have
\begin{align}\label{overlinep-asym-2}
\overline{p}(n)=\alpha(n)e^{\mu}+E_{\overline{p}}(n),
\end{align}
where
\begin{align*}
\alpha(n)=\frac{1}{8n}\left(1-\frac{1}{\mu}\right),
\end{align*}
and
\begin{align*}
|E_{\overline{p}}(n)|\leq\frac{5e^{\mu/3}}{2n^{3/2}}.
\end{align*}
\end{lem}

\pf By \eqref{overlinep-asym-1} and \eqref{R_2(n,3)}, we obtain that
\begin{align}\label{|E_{overline{p}}(n)|}
|E_{\overline{p}}(n)|
&\leq\frac{e^{-\mu}}{8n}\left(1+\frac{1}{\mu}\right)
+\frac{3^{\frac{5}{2}}}{n\mu}\sinh\left(\frac{\mu}{3}\right).
\end{align}
Define
\[g(n)=\frac{e^{-\mu}}{8n}\left(1+\frac{1}{\mu}\right).\]
Clearly, $g(n)$ is monotonically decreasing for $n>0$. For $n\geq1$,
\[g(n)<g(1)=\frac{(1+\pi)e^{-\pi}}{8\pi}<0.0072.\]
Making use of the fact that
  \[\sinh(x)<e^x/2, \quad\text{for}\, x>0,\]
we see that
\begin{align}\label{ineqq1}
|E_{\overline{p}}(n)|
&\leq0.0072+\frac{3^{\frac{5}{2}}e^{\frac{\mu}{3}}}
{2n\mu}.
\end{align}
Leting
\[f(x)=\frac{e^{\frac{\pi\sqrt{x}}{3}}}{x^{3/2}}\left(\frac{5}{2}-\frac{3^{\frac{5}{2}}}{2\pi}\right),\]
we find that for $x>1$, the minimum of $f(x)$ is at $x=81/\pi^2\approx8.21$, and $f\left(81/\pi^2\right)>0.016$, hence
we have
\begin{align}\label{ineqq2}
  \frac{5}{2n^{3/2}}e^{\frac{\mu}{3}}
-\frac{3^{\frac{5}{2}}e^{\frac{\mu}{3}}}{2n\mu}>0.0072,\quad \text{for}\, n\geq1.
\end{align}
The proof follows from \eqref{ineqq1} and \eqref{ineqq2}. \qed

Using the estimate of the overpartition function in Lemma \ref{overlinep-asym-lem}, we are ready to give a proof of the first part of Theorem \ref{overlinep-ineq-thm}.
\begin{proof}[Proof of Theorem \ref{overlinep-ineq-thm} (1)] We already know that the sequence $\overline{p}(n)$ satisfied \eqref{overlinep-log-con}. It is known that log-concave implies strong log-concavity, that is
\[\overline{p}(k)\overline{p}(\ell)\leq\overline{p}(\ell-i)\overline{p}(k+i),\]
for all $0\leq k\leq\ell\leq n$ and $0\leq i\leq\ell-k$. In particular, we take $k=n-m$, $\ell=n+m$ and $i=m$ in the above inequatity to obtain
\[\overline{p}(n)^2-\overline{p}(n-m)\overline{p}(n+m)>0,\]
for all $n>m>1$ with $n-m>1$.

Now we  consider the case $n>m>1$ with $n=m+1$. It  suffices to show that
\begin{align}\label{overlinep(a)p(b)}
\overline{p}(m+1)^2>\overline{p}(1)\overline{p}(2m+1),
\end{align}
for all $m\geq 2$.
Taking logarithms in the inequality above, we see that it is equivalent to prove that
\begin{align}\label{overlinep-m+1}
2\log\overline{p}(m+1)-\log\overline{p}(1)-\log\overline{p}(2m+1)>0,
\end{align}
for all $m\geq 2$.
Moreover, it follows from Lemma \ref{overlinep-asym-lem} that for $m\geq4$,
\begin{align}\label{overlinep-asym-3}
\frac{1}{8m}\left(1-\frac{2}{\mu(m)}\right)e^{\mu(m)}<\overline{p}(m)
<\frac{1}{8m}\left(1+\frac{1}{\mu(m)}\right)e^{\mu(m)}.
\end{align}
Combining \eqref{overlinep-asym-3} with \eqref{overlinep-m+1}, we deduce that
\begin{align*}
&-2\log(8m+8)+2\log\left(1-\frac{2}{\mu(m+1)}\right)+2\mu(m+1)-\log2+\log(16m+8)
\\
&\quad\quad-\log\left(1+\frac{1}{\mu(2m+1)}\right)-\mu(2m+1)>0,
\end{align*}
for all $m\geq4$. It is checked directly that \eqref{overlinep(a)p(b)} holds for the cases  $m=2$ and $3$.
\end{proof}

Next we will prove the second part of Theorem \ref{overlinep-ineq-thm} due to Engel's bound
\begin{align*}
\overline{p}(n)=\frac{1}{8n}\left[\left(1+\frac{1}{\mu}\right)e^{-\mu}+\left(1-\frac{1}{\mu}\right)e^{\mu}\right]
+R_2(n,3),
\end{align*}
where
\begin{align*}
|R_2(n,3)|\leq \frac{3^{\frac{5}{2}}}{n\mu}
\sinh\left(\frac{\mu}{3}\right).
\end{align*}

\begin{proof}[Proof of Theorem \ref{overlinep-ineq-thm} (2)]
We shall modify the bound of $R_2(n,N)$ slightly,
\begin{align*}
|R_2(n,N)|&\leq \sum_{m=1}^\infty\frac{4m}{4m-3}
\frac{\left(\frac{\mu(n)}{N}\right)^{2m}}{(2m+1)!}\frac{N^{3/2}}{4n}
\\[5pt]
&\leq \frac{N^{3/2}}{n}\sum_{m=1}^\infty
\frac{\left(\frac{\mu(n)}{N}\right)^{2m}}{(2m+1)!}
\\[5pt]
&=\frac{N^{5/2}}{n\mu}\left(\sinh\left(\frac{\mu}{N}\right)-\frac{\mu}{N}\right).
\end{align*}
For $N=3$, we have
\begin{align}\label{bound-R_2(n,3)}
|R_2(n,3)|\leq
\frac{3^{5/2}}{n\mu}\left(\sinh\left(\frac{\mu}{3}\right)-\frac{\mu}{3}\right)
\leq \frac{3^{\frac{5}{2}}}{n\mu}
\left[\sinh\left(\frac{\mu}{3}\right)-1\right].
\end{align}
Thanks to this error bound \eqref{bound-R_2(n,3)}, we obtain the upper bound of $\overline{p}(n)$
\begin{align}
\overline{p}(n)<\frac{e^{\mu}}{8n}\left(1+\frac{1}{n}\right),\quad \text{for}\, n\geq1.
\end{align}
On the other hand, it follows from \eqref{overlinep-asym-3} that the lower bound of $\overline{p}(n)$ is
\begin{align*}
\overline{p}(n)>\frac{e^{\mu}}{8n}\left(1-\frac{1}{\sqrt{n}}\right),
~~\text{for}~n\geq1.
\end{align*}
We may assume $1<a\leq b$, for convenience, we let $b=\lambda a$, where $\lambda\geq1$. These inequalities immediately give
\begin{align*}
\overline{p}(a)\overline{p}(\lambda a)&>
\frac{e^{\mu(a)+\mu(\lambda a)}}{64\lambda a^2}
\left(1-\frac{1}{\sqrt{a}}\right)\left(1-\frac{1}{\sqrt{\lambda a}}\right),
\\[5pt]
\overline{p}(a+\lambda a)&<
\frac{e^{\mu(a+\lambda a)}}{8a(\lambda+1)}\left(1+\frac{1}{a+\lambda a}\right).
\end{align*}
For all but finitely many cases, it suffices to find conditions on $a>1$ and $\lambda\geq1$ for which
\begin{align*}
\frac{e^{\mu(a)+\mu(\lambda a)}}{64\lambda a^2}
\left(1-\frac{1}{\sqrt{a}}\right)\left(1-\frac{1}{\sqrt{\lambda a}}\right)
>\frac{e^{\mu(a+\lambda a)}}{8a(\lambda+1)}\left(1+\frac{1}{a+\lambda a}\right).
\end{align*}
Since $\lambda\geq1$, we have that $\lambda/(\lambda+1)\geq1/2$, hence it suffices to consider when
\begin{align*}
e^{\mu(a)+\mu(\lambda a)-\mu(a+\lambda a)}>
4aS_a(\lambda),
\end{align*}
where
\begin{align}\label{S_a(lambda)}
S_a(\lambda)=\frac{1+\frac{1}{a+\lambda a}}
{\left(1-\frac{1}{\sqrt{a}}\right)\left(1-\frac{1}{\sqrt{\lambda a}}\right)}.
\end{align}
By taking the logarithm, we obtain the inequality
\begin{align}\label{T-S-ineq}
T_a(\lambda)>\log(4a)+\log(S_a(\lambda)),
\end{align}
where
\begin{align}\label{T_a(lambda)}
T_a(\lambda)=\pi\left(\sqrt{a}+\sqrt{\lambda a}-\sqrt{a+\lambda a}\right).
\end{align}
We consider \eqref{S_a(lambda)} and \eqref{T_a(lambda)} as functions in $\lambda\geq1$ and fixed $a>1$.
By simple calculations, we find  that $S_a(\lambda)$ is decreasing in $\lambda\geq1$, while $T_a(\lambda)$ is increasing in $\lambda\geq1$. Therefore, \eqref{T-S-ineq} becomes
\[T_a(\lambda)\geq T_a(1)>\log(4a)+\log(S_a(1))
\geq\log(4a)+\log(S_a(\lambda)).\]
By evaluating $T_a(1)$ and $S_a(1)$ directly, one easily finds that \eqref{T-S-ineq} holds whenever $a\geq6$. To complete the proof, assume that $2\leq a\leq 5$. We then directly calculate the real number $\lambda_a$ for which
\[T_a(\lambda_a)=\log(4a)+\log(S_a(\lambda_a)).\]
By the discussion above, if $b=\lambda a\geq a$ is an integer for which $\lambda>\lambda_a$, then \eqref{T-S-ineq} holds, which in turn gives the theorem in these cases. Table \ref{lambda_a} gives the numerical calculations for these $\lambda_a$.
 \begin{table}[!hbp]
\[\begin{array}{|c|c|}\hline
a& \lambda_a\\[2pt]\hline
2&7.578\ldots\\[2pt]\hline
3&2.566\ldots\\[2pt]\hline
4&1.550\ldots\\[2pt]\hline
5&1.117\ldots\\[2pt]\hline
\end{array}\]\caption{Values of $\lambda_a$}
\label{lambda_a}
\end{table}
Only finitely many cases remain, namely the pairs of integers where $2\leq a\leq5$ and $1\leq b/a\leq \lambda_a$. We compute $\overline{p}(a)$, $\overline{p}(b)$ and $\overline{p}(a+b)$ for these cases to complete the proof.
\end{proof}

\section{Proof of Theorem \ref{turan-ineq}}

In this section, we employ the method of Chen, Jia and Wang \cite{Chen-Jia-Wang-2017}, which is used to the third order Tur\'{a}n inequality for the partition function, to prove the third order Tur\'{a}n inequality for the overpartition function
\begin{align*}
4(1-u_n)(1-u_{n+1})-(1-u_nu_{n+1})^2>0,\quad \text{for}\, n\geq16.
\end{align*}
To this end,  we first bound the ratio $u_n=\overline{p}(n-1)\overline{p}(n+1)/\overline{p}(n)^2$.
Then we build some inequalities  among $\mu=\mu(n)=\pi\sqrt{n}$ and the lower bound $f(n)$ and the upper bound $g(n)$ for $u_n$.
Finally,  the distribution of the roots of the polynomial $F(t)=4(1-u_n)(1-t)-(1-u_nt)^2$ gives us the chance to prove the third order Tur\'{a}n inequality for the overpartition function.

 Next we  find an effective bound for the overpartition function $\overline{p}(n)$ and  then give the upper and lower bounds of $u_n$,
\begin{thm}\label{u_bounds}
For $n\geq55$,
\begin{align}\label{turan-u_n}
f(n)<u_n<g(n),
\end{align}
where
\[x=\mu(n-1),~~~y=\mu=\mu(n),~~~z=\mu(n+1),~~~w=\mu(n+2),\]
and
\begin{align}\label{turan-f}
f(n)&=e^{x-2y+z}\frac{y^{14}(x^5-x^4-1)(z^5-z^4-1)}{x^7z^7(y^5-y^4+1)^2},
\\[5pt] \label{turan-g}
g(n)&=e^{x-2y+z}\frac{y^{14}(x^5-x^4+1)(z^5-z^4+1)}{x^7z^7(y^5-y^4-1)^2}.
\end{align}
\end{thm}

\begin{proof}
Let
\begin{align*}
B_1(n)&=\frac{e^{\mu}}{8n}\left(1-\frac{1}{\mu}-\frac{1}{\mu^5}\right),
\\
B_2(n)&=\frac{e^{\mu}}{8n}\left(1-\frac{1}{\mu}+\frac{1}{\mu^5}\right).
\end{align*}
We first claim that the following bounds for the overpartition function $\overline{p}(n)$ holds,
\begin{align}\label{overp-bound}
B_1(n)<\overline{p}(n)<B_2(n),~~~\text{for}~n\geq 55.
\end{align}
Set
\begin{align*}
\widetilde{T}(n)&=\left(1+\frac{1}{\mu}\right)e^{-2\mu}+\frac{8n}{e^{\mu}}R_2(n,3).
\end{align*}
So we can rewrite \eqref{overlinep-asym-1} as
\begin{align}\label{turan-overlinep-asym}
\overline{p}(n)=\frac{e^{\mu}}{8n}\left(1-\frac{1}{\mu}+\widetilde{T}(n)\right),
\end{align}
where
\begin{align*}
|R_2(n,3)|\leq \frac{3^{\frac{5}{2}}}{n\mu}\sinh\left(\frac{\mu}{3}\right)
\leq \frac{3^{\frac{5}{2}}e^{\frac{\mu}{3}}}{2n\mu}.
\end{align*}
Obviously, for $n\geq1$,
\[0<\frac{1}{\mu}<\frac{1}{2},\]
we have
\[\left(1+\frac{1}{\mu}\right)e^{-2\mu}<2e^{-2\mu}<2e^{-\frac{2}{3}\mu}.\]
As for the last term in $\widetilde{T}(n)$,
\begin{align*}
\left|\frac{8n}{e^{\mu}}R_2(n,3)\right|
<\left|4\cdot3^{\frac{5}{2}}\frac{e^{-\frac{2}{3}\mu}}{\mu}\right|<32e^{-\frac{2}{3}\mu}.
\end{align*}
Thus
\begin{align}\label{turan-overp-T'}
|\widetilde{T}(n)|<34e^{-\frac{2}{3}\mu}.
\end{align}
Next we aim to prove that for $n\geq143$,
\begin{align}\label{turan-overp-r}
34e^{-\frac{2}{3}\mu}<\frac{1}{\mu^5},
\end{align}
which can be recast as
\[\frac{e^{2\mu/15}}{2\mu/15}>\frac{15}{2}\cdot\sqrt[5]{34}.\]
Let
$F(t)=e^t/t$.
Since $F'(t)=e^t(t-1)/t^2>0$ for $t>1$, $F(t)$ is increasing for $t>1$.
Observe that for $n>142$, $2\mu/15>5$.
Thus,
\[F\left(\frac{2\mu}{15}\right)=\frac{e^{2\mu/15}}{2\mu/15}
>F(4)=\frac{e^5}{5}>\frac{15}{2}\sqrt[5]{34}.\]
So \eqref{turan-overp-r} holds for $n\geq143$. Thus, combining \eqref{turan-overp-T'} and \eqref{turan-overp-r}, we get that for $n\geq143$,
\begin{align}\label{turan-overp-T'-1}
-\frac{1}{\mu^5}<\widetilde{T}(n)<\frac{1}{\mu^5}.
\end{align}
Substituting \eqref{turan-overp-T'-1} into \eqref{turan-overlinep-asym}, we see that \eqref{overp-bound} holds for $n\geq143$. It is routine to check that \eqref{overp-bound} is true for $55\leq n\leq142$, and hence the claim \eqref{overp-bound} can be verified.

Since $B_1(n)$ and $B_2(n)$ are all positive for $n\geq1$, using the bounds for $\overline{p}(n)$ in \eqref{overp-bound}, we find that for $n\geq55$,
\[\frac{B_1(n-1)B_1(n+1)}{B_2(n)^2}<\frac{\overline{p}(n-1)\overline{p}(n+1)}{\overline{p}(n)^2}
<\frac{B_2(n-1)B_2(n+1)}{B_1(n)^2},\]
and which completes the proof.
\end{proof}


Now we will build an inequality between $f(n)$ and $g(n+1)$.
\begin{thm}\label{Turan_f_g_bound}
For $n\geq2$,
\begin{align}\label{turan-f-g-bound}
 g(n+1)<f(n)+\frac{1000}{\mu(n-1)^5}.
\end{align}
\end{thm}

\pf
 Recall that
\begin{align*}
\mu(n)=\pi\sqrt{n},
\end{align*}
and
\begin{align*}
x=\mu(n-1),~~~y=\mu(n),~~~z=\mu(n+1),~~~w=\mu(n+2).
\end{align*}
Let
\begin{align*}
\alpha(t)=t^5-t^4+1,~~~\beta(t)=t^5-t^4-1.
\end{align*}
By \eqref{turan-f} and \eqref{turan-g}, we see that
\begin{align*}
f(n)x^5-g(n+1)x^5+1000=\frac{-e^{w+y-2z}t_1+e^{z+x-2y}t_2+1000t_3}{t_3},
\end{align*}
where
\begin{align}\label{turan-t1}
t_1&=x^7z^{21}\alpha(y)^3\alpha(w),
\\[5pt]\label{turan-t2}
t_2&=y^{21}w^7\beta(x)\beta(z)^3,
\\[5pt]\label{turan-t3}
t_3&=x^2y^7z^7w^7\alpha(y)^2\beta(z)^2.
\end{align}
Since $t_3>0$ for $n\geq2$, \eqref{turan-f-g-bound} is equivalent to
\begin{align*}
-e^{w+y-2z}t_1+e^{z+x-2y}t_2+1000t_3>0
\end{align*}
for $n\geq2$. To do this, we need to estimate $t_1$, $t_2$, $t_3$, $e^{w+y-2z}$ and $e^{x-2y+z}$ in terms of $x$. Note that for $n\geq2$,
\begin{align*}
y=\sqrt{x^2+\pi^2},~~~z=\sqrt{x^2+2\pi^2},~~~w=\sqrt{x^2+3\pi^2},
\end{align*}
Then for $x>1$, we have the following expansions
\begin{align*}
y&=x+\frac{\pi^2}{2x}-\frac{\pi^4}{8x^3}+\frac{\pi^6}{16x^5}-\frac{5\pi^8}{128x^7}+\frac{7\pi^{10}}{256x^9}
-\frac{21\pi^{12}}{1024x^{11}}+O\left(\frac{1}{x^{12}}\right),
\nonumber\\[5pt]
z&=x+\frac{\pi^2}{x}-\frac{\pi^4}{2x^3}+\frac{\pi^6}{2x^5}-\frac{5\pi^8}{8x^7}+\frac{7\pi^{10}}{8x^9}
-\frac{21\pi^{12}}{16x^{11}}+O\left(\frac{1}{x^{12}}\right),
\nonumber\\[5pt]
w&=x+\frac{3\pi^2}{2x}-\frac{9\pi^4}{8x^3}+\frac{27\pi^6}{16x^5}-\frac{405\pi^8}{128x^7}
+\frac{1701\pi^{10}}{256x^9}-\frac{15309\pi^{12}}{1024x^{11}}+O\left(\frac{1}{x^{12}}\right).
\end{align*}
It is easy to  see that for $x>1$,
\begin{align}\label{turan-y-bound}
&y_1<y<y_2,
\\[3pt] \label{turan-z-bound}
&z_1<z<z_2,
\\[3pt] \label{turan-w-bound}
&w_1<w<w_2,
\end{align}
where
\begin{align*}
&y_1=x+\frac{\pi^2}{2x}-\frac{\pi^4}{8x^3}+\frac{\pi^6}{16x^5}-\frac{5\pi^8}{128x^7}+\frac{7\pi^{10}}{256x^9}
-\frac{21\pi^{12}}{1024x^{11}},
\\[5pt]
&y_2=x+\frac{\pi^2}{2x}-\frac{\pi^4}{8x^3}+\frac{\pi^6}{16x^5}-\frac{5\pi^8}{128x^7}+\frac{7\pi^{10}}{256x^9},
\\[5pt]
&z_1=x+\frac{\pi^2}{x}-\frac{\pi^4}{2x^3}+\frac{\pi^6}{2x^5}-\frac{5\pi^8}{8x^7}+\frac{7\pi^{10}}{8x^9}
-\frac{21\pi^{12}}{16x^{11}},
\\[5pt]
&z_2=x+\frac{\pi^2}{x}-\frac{\pi^4}{2x^3}+\frac{\pi^6}{2x^5}-\frac{5\pi^8}{8x^7}+\frac{7\pi^{10}}{8x^9},
\\[5pt]
&w_1=x+\frac{3\pi^2}{2x}-\frac{9\pi^4}{8x^3}+\frac{27\pi^6}{16x^5}-\frac{405\pi^8}{128x^7}
+\frac{1701\pi^{10}}{256x^9}-\frac{15309\pi^{12}}{1024x^{11}},
\\[5pt]
&w_2=x+\frac{3\pi^2}{2x}-\frac{9\pi^4}{8x^3}+\frac{27\pi^6}{16x^5}-\frac{405\pi^8}{128x^7}
+\frac{1701\pi^{10}}{256x^9}.
\end{align*}

 Next we make use of these bounds of $y$, $z$ and $w$ in \eqref{turan-y-bound}, \eqref{turan-z-bound} and \eqref{turan-w-bound} to estimate $t_1$, $t_2$, $t_3$, $e^{w+y-2z}$ and $e^{x-2y+z}$ in terms of $x$.

First, we give estimates for $t_1$, $t_2$ and $t_3$.
We use \eqref{turan-w-bound} to derive that for $x>1$,
\[w_1w^4<w^5<w_2w^4.\]
Let
\[\eta_1=w_2w^4-w^4+1,\]
so that for $x>1$,
\begin{align}\label{turan-eta-1}
\alpha(w)<\eta_1.
\end{align}
Similarly, set
\begin{align*}
\eta_2&=y_2y^{14}-3y^{14}+3y_2y^{12}-y^{12}+3y^{10}-6y_1y^8+3y^8+3y_2y^4-3y^4+1,
\\[3pt]
\eta_3&=z_1z^{14}-3z^{14}+3z_1z^{12}-z^{12}-3z^{10}+6z_1z^8-3z^8+3z_1z^4-3z^4-1,
\\[3pt]
\eta_4&=y^{10}-2y_2y^8+y^8+2y_1y^4-2y^4+1,
\\[3pt]
\eta_5&=z^{10}-2z_2z^8+z^8-2z_2z^4+2z^4+1.
\end{align*}
Then we have for $x>1$,
\begin{align}\label{turan-eta-2}
\alpha(y)^3<\eta_2,~~~\beta(z)^3>\eta_3,~~~\alpha(y)^2>\eta_4,~~~\beta(z)^2>\eta_5.
\end{align}
Together the relations in \eqref{turan-eta-1} and \eqref{turan-eta-2}, we find that for $x>1$,
\begin{align}\label{turan-bound-t1}
t_1&=x^7z^{21}\alpha(y)^3\alpha(w)<x^7z_2z^{20}\eta_1\eta_2,
\\[5pt]\label{turan-bound-t2}
t_2&=y^{21}w^7(x^5-x^4-1)\beta(z)^3>y_1y^{20}w_1w^6(x^5-x^4-1)\eta_3,
\\[5pt]\label{turan-bound-t3}
t_3&=x^2y^7z^7w^7\alpha(y)^2\beta(z)^2>x^2y_1y^6z_1z^6w_1w^6\eta_4\eta_5.
\end{align}

We continue to estimate $e^{w+y-2z}$ and $e^{z+x-2y}$. Applying \eqref{turan-y-bound}, \eqref{turan-z-bound} and \eqref{turan-w-bound} to $w+y-2z$, we see that for $x>1$,
\begin{align}\label{turan-bound-w+y-2z}
w+y-2z<w_2+y_2-2z_1,
\end{align}
which implies that
\begin{align}\label{turan-bound-e-w+y-2z}
e^{w+y-2z}<e^{w_2+y_2-2z_1}.
\end{align}

 We define
\begin{align}\label{turan-Phi}
\Phi(t)=1+t+\frac{t^2}{2}+\frac{t^3}{6}+\frac{t^4}{24}+\frac{t^5}{120}+\frac{t^6}{720},
\end{align}
so as to give a feasible upper bound for $e^{w+y-2z}$,
Then we have that for $t<0$,
\begin{align}\label{e_Phi_ineq}
e^t<\Phi(t).
\end{align}

Since  $\pi^4(16x^8-48\pi^2x^6>0$ and $125\pi^4x^4-315\pi^6x^2-168\pi^8>0$ both hold for $x\geq 6$,
\[w_2+y_2-2z_1=-\frac{\pi^4(16x^8-48\pi^2x^6+125\pi^4x^4-315\pi^6x^2-168\pi^8)}{64x^{11}}<0\]
holds for $x\geq6$.
Thus, we deduce that for $x\geq6$
\begin{align}\label{turan-bound-Phi-w+y-2z}
e^{w_2+y_2-2z_1}<\Phi(w_2+y_2-2z_1).
\end{align}
Then it follows from \eqref{turan-bound-e-w+y-2z} and \eqref{turan-bound-Phi-w+y-2z}  that for $x\geq6$,
\begin{align}\label{turan-bound-e-Phi-w+y-2z}
e^{w+y-2z}<\Phi(w_2+y_2-2z_1).
\end{align}
Similarly, applying \eqref{turan-y-bound}, \eqref{turan-z-bound} and \eqref{turan-w-bound} to $z+x-2y$, we find that for $x>1$,
\begin{align}\label{turan-bound-z+x-2y}
z_1+x-2y_2<z+x-2y,
\end{align}
so that
\begin{align}\label{turan-bound-e-z+x-2y}
e^{z_1+x-2y_2}<e^{z+x-2y}.
\end{align}
Define
\begin{align}\label{turan-phi}
\phi(t)=1+t+\frac{t^2}{2}+\frac{t^3}{6}+\frac{t^4}{24}+\frac{t^5}{120}+\frac{t^6}{720}
+\frac{t^7}{5040}.
\end{align}
It can be easily verified that for $t<0$,
$\phi(t)<e^t$.
Since
\begin{align*}
z+x-2y&=\sqrt{x^2+2\pi^2}+x-2\sqrt{x^2+\pi^2}
\\[5pt]
&=-\frac{\left(\sqrt{x^2+2\pi^2}-x\right)^2}{\sqrt{x^2+2\pi^2}+x+2\sqrt{x^2+\pi^2}}<0
\end{align*}
for $x\geq5$, we deduce that for $x\geq5$,
\begin{align*}
z_1+x-2y_2<0.
\end{align*}
Thus, we get that for $x\geq5$,
\begin{align}\label{turan-bound-phi-z+x-2y}
\phi(z_1+x-2y_2)<e^{z_1+x-2y_2}.
\end{align}
Combining \eqref{turan-bound-e-z+x-2y} and \eqref{turan-bound-phi-z+x-2y} yields that for $x\geq5$,
\begin{align}\label{turan-bound-e-phi-z+x-2y}
e^{z+x-2y}>\phi(z_1+x-2y_2).
\end{align}
Using the above bounds for $t_1$, $t_2$, $t_3$, $e^{w+y-2z}$ and $e^{x-2y+z}$, we obtain that for $x\geq6$,
\begin{align*}
&-e^{w+y-2z}t_1+e^{z+x-2y}t_2+1000t_3
\\[3pt]
&>-\Phi(w_2+y_2-2z_1)x^7z_2z^{20}\eta_1\eta_2+\phi(z_1+x-2y_2)y_1w_1y^{20}w^6(x^5-x^4-1)\eta_3
\\[3pt]
&\quad\quad+1000x^2y_1z_1w_1y^6z^6w^6\eta_4\eta_5.
\end{align*}
It remains to verify that for $x\geq5$,
\begin{align*}
&-\Phi(w_2+y_2-2z_1)x^7z_2z^{20}\eta_1\eta_2+\phi(z_1+x-2y_2)y_1w_1y^{20}w^6(x^5-x^4-1)\eta_3
\\[3pt]
&\quad\quad+1000x^2y_1z_1w_1y^6z^6w^6\eta_4\eta_5>0.
\end{align*}
Replacing $y$, $z$ and $w$ by $\sqrt{x^2+\pi^2}$, $\sqrt{x^2+2\pi^2}$ and $\sqrt{x^2+3\pi^2}$ respectively, we see that the left hand side of above inequality can be expressed as $H(x)/G(x)$, where
\begin{align*}
H(x)=\sum_{k=0}^{153}a_kx^k
\end{align*}
and
\[G(x)=47601454147326023754055680 x^{110}.\]
Here we just list the last few values of
\begin{align*}
a_{151}&=1487545442103938242314240
\\[6pt]
&\quad\quad\times
\left(191232+1143744\pi^2-388\pi^6-387\pi^8\right),
\\[6pt]
a_{152}&=166605089515641083139194880
\left(-1136+\pi^6\right),
\\[6pt]
a_{153}&=5950181768415752969256960
\left(7936-3\pi^6\right),
\end{align*}
which $a_{151}$ and $a_{153}$ are positive, but $a_{152}$ is negative.

Becasue $G(x)$ is always positive for all positive $x$, it suffices to prove that $H(x)>0$. It is clear that  $x\geq2$ for $n\geq2$ and hence
\begin{align*}
H(x)\geq\sum_{k=0}^{150}-|a_k|x^k+a_{151}x^{151}+a_{152}x^{152}+a_{153}x^{153}.
\end{align*}
Moreover, numerical evidence indicates that for any $0\leq k\leq 150$,
\begin{align*}
-|a_k|x^k>-a_{151}x^{151}
\end{align*}
holds for $x\geq14$. It follows that for $x\geq14$,
\begin{align*}
\sum_{k=0}^{150}-|a_k|x^k+a_{152}x^{152}+a_{153}x^{153}
>-151a_{151}x^{151}+a_{151}x^{151}+a_{152}x^{152}+a_{153}x^{153},
\end{align*}
which yields that
\begin{align*}
H(x)>\left(-150a_{151}+a_{152}x+a_{153}x^2\right)x^{151}.
\end{align*}
Thus, $H(x)$ is positive provided
\begin{align*}
-150a_{151}+a_{152}x+a_{153}x^2>0,
\end{align*}
which is true if
\[x>\frac{-a_{152}+\sqrt{a_{152}^2+600a_{151}a_{153}}}{2a_{153}}\approx 235.402.\]
So we conclude that $H(x)$ is positive if $x\geq236$.
Therefore, for $x\geq236$, or equivalently, for $n\geq5615$,
\begin{align}\label{turan-e-ineq}
-e^{w+y-2z}t_1+e^{z+x-2y}t_2+1000t_3>0.
\end{align}
For $2\leq n\leq5614$, \eqref{turan-e-ineq} can be directly verified. So we complete the proof. \qed

The following result is an inequality on $u_n$ and $f(n)$ and is also an important step to prove the third Tur\'{a}n inequality in Theorem \ref{turan-ineq}.
\begin{thm}\label{Q_f_ineq}
For $0<t<1$, let
\begin{align}\label{turan-Q}
Q(t)=\frac{3t+2\sqrt{(1-t)^3}-2}{t^2}.
\end{align}
Then for $n\geq92$,
\begin{align}\label{turan-u-f-bound}
f(n)+\frac{1000}{\mu(n-1)^5}<Q(u_n).
\end{align}
\end{thm}

Before we give a proof of Theorem \ref{Q_f_ineq}, we need the following lemma.
Recall that
\begin{align*}
f(n)=e^{x-2y+z}\frac{y^{14}(x^5-x^4-1)(z^5-z^4-1)}{x^7z^7(y^5-y^4+1)^2}
\end{align*}
and
\[  \Phi(t)=1+t+\frac{t^2}{2}+\frac{t^3}{6}+\frac{t^4}{24}+\frac{t^5}{120}+\frac{t^6}{720}.
\]
\begin{lem}
For $n\geq 4$, we have
\begin{align}\label{f(n)_upper_bound}
  f(n)<\frac{\Phi(x-2y_1+z_2)(x^5-x^4-1)y^{14}(z_2z^4-z^4-1)}{x^7(y^{10}-2y_2y^8+y^8+2y_1y^4-2y^4+1)z_1z^6}
  <1,
\end{align}
where $y_1,y_2,z_1$ and $z_2$ are defined in the proof of Theorem \ref{Turan_f_g_bound}.
\end{lem}
\begin{proof}
From \eqref{turan-y-bound} and \eqref{turan-z-bound} we see that for $x\geq1$,
\begin{align}
  e^{x-2y+z}&<e^{x-2y_1+z_2},\label{ineq1}\\[5pt]
  z^5-z^4-1&<z_2z^4-z^4-1,\label{ineq2}\\[5pt]
  (y^5-y^4+1)^2>y^{10}-2y_2y^8&+y^8+2y_1y^4-2y^4+1.\label{ineq3}
\end{align}
Now we give an upper bound for $e^{x-2y_1+z_2}$. Notice that
\begin{align}
  x-2y_1+z_2=-\frac{\pi^4 \left(128 x^8-192 \pi^2 x^6+280 \pi^4 x^4-420 \pi^6 x^2-21 \pi^8\right)}{512 x^{11}}.
\end{align}
Moreover, It is easily verified that
 \begin{align*}
  128 x^8-192 \pi^2 x^6>0,\quad \text{for}\quad x\geq4,
 \end{align*}
 and
  \begin{align*}
  280 \pi^4 x^4-420 \pi^6 x^2-21 \pi^8>0, \quad \text{for} \quad x\geq4.
 \end{align*}
 Therefore, $x-2y_1+z_2<0$ holds for $x\geq 4$. It follows from  \eqref{e_Phi_ineq} that for $x\geq 4$,
 \begin{align}\label{ineq4}
  e^{x-2y_1+z_2}<\Phi(x-2y_1+z_2).
 \end{align}
 Combining \eqref{ineq1} with \eqref{ineq4}, we find that for $x\geq 4$,
 \begin{align}\label{ineq5}
   e^{x-2y+z}<\Phi(x-2y_1+z_2).
 \end{align}
 Together with \eqref{ineq2}, \eqref{ineq3} and \eqref{ineq5}, we see that the first inequality in \eqref{f(n)_upper_bound} holds for $x\geq 4$, or equivalently, $n\geq 2$.

 To prove the second inequality in \eqref{f(n)_upper_bound}, we define the polynomial $H(x)$ and $G(x)$ to be the numerator and denominator of
 \[  \frac{\Phi(x-2y_1+z_2)(x^5-x^4-1)y^{14}(z_2z^4-z^4-1)}{x^7(y^{10}-2y_2y^8+y^8+2y_1y^4-2y^4+1)z_1z^6},
\]
respectively.
It is easy to see that  $H(x)$ and $G(x)$ are both polynomials of degree $99$.
For convenience, write
\begin{align}\label{H_G_Poly}
H(x)=\sum_{k=0}^{99}b_kx^k, G(x)=\sum_{k=0}^{99}c_kx^k.
\end{align}
Here are the values of $b_k$ and $c_k$ for $94\leq k \leq99$:
\begin{align*}
 b_{94}&=-2^{58}\cdot3^2\cdot5\cdot(16+934 \pi ^4+21 \pi ^6),\\[5pt]
 b_{95}&=2^{61}\cdot3^2\cdot5\cdot\pi^2\cdot(11+64 \pi ^2),\\[5pt]
 b_{96}&=-2^{59}\cdot3^2\cdot5\cdot\pi^2\cdot(92+\pi^2),\\[5pt]
  c_{94}&=2^{59}\cdot3^2\cdot5\cdot(8-455 \pi ^4),\\[5pt]
 c_{95}&=2^{60}\cdot3^2\cdot5\cdot\pi^2\cdot(22+125 \pi ^2),\\[5pt]
 c_{96}&=-2^{61}\cdot3^2\cdot5\cdot23\cdot\pi^2,\\[5pt]
 b_{97}&=c_{97}=2^{61}\cdot3^2\cdot5\cdot(1+12\pi ^2),\\[5pt]
 b_{98}&=c_{98}=-2^{62}\cdot3^2\cdot5,\\[5pt]
 b_{99}&=c_{99}=2^{61}\cdot3^2\cdot5.
\end{align*}
In order to complete the proof of this lemma , it suffices to show that for $x\geq 8$,
\begin{align}\label{G_nonnegative}
 G(x)>0,
\end{align}
and
 \begin{align}\label{G_H_nonnegative}
 G(x)-H(x)>0.
\end{align}
If \eqref{G_nonnegative} and \eqref{G_H_nonnegative} is verified, we see that the second inequality in \eqref{f(n)_upper_bound} holds for $x\geq 109$, or equivalently, $n\geq 1204$.
The cases for $4\leq n\leq 1204$ can be directly verified, and the proof follows.

Thus it remains to verify \eqref{G_nonnegative} and \eqref{G_H_nonnegative}.
Simple calculations reveal that  for $0\leq k\leq 96$,
\begin{align}\label{c_k_ineq}
  -|c_k|x^k>-c_{97}x^{97}
\end{align}
holds when
\[x>\pi\sqrt{\frac{22+125\pi^2}{2(1+12\pi^2)}}\approx 7.203.\]
Then it follows that for $x\geq 8$,
\[G(x)>-96c_{97}x^{97}+c_{98}x^{98}+c_{99}x^{99}.\]
Since \[-96c_{97}+c_{98}x+c_{99}x^{2}>0\]
for $x>\sqrt{97+1152 \pi^2}+1\approx108.083$, we have
$G(x)>0$ for $x\geq 109$.

Now we turn to prove \eqref{G_H_nonnegative}. It is easy to check that
for $0\leq k\leq 93$,
\[-|c_k-b_k|x^k>-(c_{94}-b_{94})x^{94}\]
for $x>\frac{\pi}{2} \sqrt{\frac{2432+1824\pi^4+767 \pi^6}{2 \left(32+24 \pi^4+21 \pi^6\right)}}\approx7.083$. It immediately follows that
\[G(x)-H(x)>\left(-93(c_{94}-b_{94})+(c_{95}-b_{95})x+(c_{96}-b_{96})x^2\right)x^{94}.\]
Moreover, we find that
for $x>\frac{\sqrt{\frac{3}{2} \left(992+750 \pi ^4+651 \pi ^6\right)}}{\pi ^2}+3\approx106.817$,
\[-93(c_{94}-b_{94})+(c_{95}-b_{95})x+(c_{96}-b_{96})x^2>0.\]
Thus, for $x\geq107$, $G(x)-H(x)>0$.
\end{proof}

We are now ready to prove Theorem \ref{Q_f_ineq}.
\begin{proof}[Proof of Theorem \ref{Q_f_ineq}]
 It is easy to see that $Q(t)$ is increasing for $0<t<1$ since
 \[Q'(t)=\frac{1}{(\sqrt{1-t}+1)^3}\]
 is positive for $0<t<1$. By Theorem \ref{u_bounds}, we know that  $f(n)<u_n$ for $n\geq 29$. Then we have for $n\geq 9$,
 \[Q(f(n))<Q(u_n).\]
 If we can prove
 \begin{align}\label{Qf_ineq}
 f(n)+\frac{1000}{\mu(n-1)^5}<Q(f(n))
 \end{align}
  for $n\geq 30985$, it is done. Let
 \[\psi(t)=Q(t)-t=\frac{3t+2\sqrt{(1-t)^3}-t^3-2}{t^2}.\]
 Then \eqref{Qf_ineq} is equivalent to
 \[\psi(f(n))>\frac{1000}{\mu(n-1)^5}.\]
 Since for $0<t<1$,
 \[\psi'(t)=\frac{\sqrt{1-t}(-t+3\sqrt{1-t}+4)}{(\sqrt{1-t}+1)^3}<0,\]
 it is clear that $\psi(t)$ is decreasing for $0<t<1$.
 From \eqref{f(n)_upper_bound} we see that $0<f(n)<H(x)/G(x)<1$ for $n\geq4$. So it remains to prove
 \[\psi\left(f(n)\right)>\psi\left(\frac{H(x)}{G(x)}\right),\quad \text{for}\quad n\geq30985.\]
 Therefore the proof is reduced to prove that for $n\geq 30985$,
 \begin{align}\label{H_G_mu_ineq3}
  \psi\left(\frac{H(x)}{G(x)}\right)>\frac{1000}{\mu(n-1)^5}.
 \end{align}
 To this end, we should give an estimate for $\psi\left(\frac{H(x)}{G(x)}\right)$.
 Firstly, we claim that  for $x\geq 109$,
\begin{align}\label{H_G_ineq}
\frac{\sqrt{5}-1}{2}<\frac{H(x)}{G(x)}<1.
\end{align}
To do this, it suffices to show that
\begin{align}\label{H_G_ineq0}
2H(x)-(\sqrt{5}-1)G(x)\geq 0, \quad \text{for}\, x\geq 109.
\end{align}
Notice that
\[b_{97}=c_{97},\quad  b_{98}=c_{98},\quad b_{99}=c_{99},\]
and observe that for $0\leq k\leq 96,$
\[-|2b_k-(\sqrt{5}-1)c_k|x^k>-(3-\sqrt{5})c_{97}x^{97}\]
when
\[x>\sqrt{\frac{-125 \pi ^2 \sqrt{5}-22 \sqrt{5}+381 \pi ^2+66}{2 \left(3-\sqrt{5}\right) \left(1+12 \pi ^2\right)}}\approx 7.42197.\]
Then it follows that for $x\geq 8$,
\[2H(x)-(\sqrt{5}-1)G(x)>(3-\sqrt{5})\left(-96 c_{97}+c_{98}x+c_{99}x^2\right)x^{97}.\]
Since $-96 c_{97}+c_{98}x+c_{99}x^2>0$
for $x>\sqrt{97+1152 \pi ^2}+1\approx108.083$, we arrive at \eqref{H_G_ineq0}, and so \eqref{H_G_ineq} holds for $x\geq 109$.

 Secondly, we find that
\begin{align}\label{psi_t_ineq}
\psi(t)<(1-t)^{3/2}, \quad \text{for any} \,\frac{\sqrt{5}-1}{2}<t<1.
\end{align}
 This is because
 \[\psi(t)-(1-t)^{3/2}=\frac{(1-t)^{3/2}(t-\frac{\sqrt{5}-1}{2})(t+\frac{\sqrt{5}-1}{2})}{(\sqrt{1-t}+1)^2(\sqrt{1-t}+t)}>0\]
for $\frac{\sqrt{5}-1}{2}<t<1$.
In view of \eqref{H_G_ineq} and \eqref{psi_t_ineq}, we infer that for $x\geq 109$,
\begin{align}\label{psi_H_G_ineq}
 \psi\left(\frac{H(x)}{G(x)}\right)>\left(1-\frac{H(x)}{G(x)}\right)^{3/2}.
\end{align}

We continue to show that for $x\geq553$, or equivalently, $n\geq 30985$,
\begin{align}\label{H_G_mu}
\left(1-\frac{H(x)}{G(x)}\right)^{3/2}>\frac{1000}{\mu(n-1)^5}.
\end{align}
Since $G(x)>0$ for $x\geq8$, the above inequality can be reformulated as follows. For $x\geq 555$,
\begin{align}\label{H_G_ineq2}
  x^{10}(G(x)-H(x))^3-1000^2G(x)^3>0.
\end{align}
The left side of \eqref{H_G_ineq2} is a polynomial of degree 298, and we write
\[ x^{10}(G(x)-H(x))^3-1000^2G(x)^3=\sum_{k=0}^{298}\gamma_k x^k.\]
The values of $\gamma_{296}$, $\gamma_{297}$ and $\gamma_{298}$ are given below:
\begin{align*}
  \gamma_{296}&=2^{176}\cdot3^7\cdot5^3\cdot(21  \pi^{14}+96  \pi^{12}+32 \pi^8+256000000),\\[5pt]
    \gamma_{297}&=-2^{178}\cdot3^6\cdot5^3\cdot(32000000+9 \pi ^{12}),\\[5pt]
      \gamma_{298}&=2^{177}\cdot3^6\cdot5^3\cdot\pi^{12}.
\end{align*}
For $0\leq k\leq 295$, we have
\[-|\gamma_{k}|x^k>-\gamma_{296}x^{296},\]
provided that
\[x>\frac{-2560000000-6144000000 \pi ^2-1664 \pi ^8-1776 \pi ^{12}-1488 \pi ^{14}-\pi ^{16}}{-1024000000-128 \pi ^8-384 \pi ^{12}-84 \pi ^{14}}\approx 36.5822.\]
Thus, for $x\geq37$,
\[x^{10}(G(x)-H(x))^3-1000^2G(x)^3>\left(-295\gamma_{296}+\gamma_{297}x+\gamma_{298}x^2\right)x^{296}.\]
The left hand side of the above inequality is positive, since
\[-295\gamma_{296}+\gamma_{297}x+\gamma_{298}x^2>0\]
when
\[x>\frac{\sqrt{\gamma_{297}^2+1180\gamma_{296}\gamma_{298}}-\gamma_{297}}{2\gamma_{298}}\approx552.349.\]
Therefore \eqref{H_G_mu} is true. Combining \eqref{psi_H_G_ineq} and \eqref{H_G_mu} yields \eqref{H_G_mu_ineq3} is true for $n\geq30985$. The proof follows from checking that
\eqref{turan-u-f-bound} is true for $92\leq n<30985$ directly.
\end{proof}


With Theorems \ref{u_bounds}, \ref{Turan_f_g_bound} and \ref{Q_f_ineq} in hand, we are ready to give a proof of Theorem \ref{turan-ineq} as follows.
\begin{proof}[Proof of Theorem \ref{turan-ineq}]
 From \eqref{exact-logconcave} we know that $u_n<1$ for $n\geq 2$.
 Define $F(t)$ to be
 \[F(t)=4(1-u_n)(1-t)-(1-u_nt)^2.\]
 Then it is easy to see that
the inequality
 \[4(1-u_n)(1-u_{n+1})-(1-u_nu_{n+1})^2>0,\quad \text{for} \, n\geq16,\]
 which is equivalent to
 \begin{align}\label{F_ineq}
  F(u_{n+1})>0,\quad \text{for} \, n\geq16.
 \end{align}
 For $16\leq n\leq 91$, \eqref{F_ineq} can be easily checked. Therefore, it remains to prove that  \eqref{F_ineq} holds for $n\geq 92$. Let $Q(t)$ be as defined in Theorem \ref{Q_f_ineq}, that is \[Q(t)=\frac{3t+2\sqrt{(1-t)^3}-2}{t^2}.\]
 Here we first claim that $F(t)>0$ for $u_n<t<Q(u_n)$. So the proof is reduced to proof that for $n\geq 92$,
 \[u_n\leq u_{n+1}\leq Q(u_n).\]
 Observe that Wang, Xie and Zhang \cite[Theorem 3.1]{Wang-Xie-Zhang-2018} proved that $u_n<u_{n+1}$ for $n\geq 18$.
From Theorem \ref{u_bounds} we know that $u_{n+1}<g(n+1)$ for $n\geq 92$.
 Moreover, combining Theorem \ref{Turan_f_g_bound} with Theorem \ref{Q_f_ineq} yields that for $n\geq 92$,
 \[g(n+1)<f(n)+\frac{1000}{\mu(n-1)^5}<Q(u_n).\]
Therefore, we conclude that $u_{n+1}<Q(u_n)$ for $n\geq 92$, as required.

Finally, it remains to verify the previous claim.
Rewrite $F(t)$ as
 \[F(t)=-u_n^2t^2+(6u_n-4)t-4u_n+3.\]
 The equation $F(t)=0$ has two solutions
 \[P(u_n)=\frac{3u_n-2\sqrt{(1-u_n)^3}-2}{u_n^2},\quad
 Q(u_n)=\frac{3u_n+2\sqrt{(1-u_n)^3}-2}{u_n^2},\]
 so that $F(t)>0$ for $P(u_n)<u_n<Q(u_n)$.
 Therefore, $F(t)>0$ for $u_n<t<Q(u_n),$ as claimed.
\end{proof}

\end{document}